\documentclass[12pt,reqno]{amsart}
\usepackage{amsmath, amssymb, amsfonts}
\numberwithin{equation}{section}
\newtheorem{thm}{Theorem}[section]
\newtheorem{prop}[thm]{Proposition}
\newtheorem{lem}[thm]{Lemma}
\newtheorem{dfn}{Definition}

\newtheorem{cor}[thm]{Corollary}

\usepackage[a4paper, top = 1.2in, bottom = 1.2in, left = 1.2in, right = 1.2in]{geometry}

\numberwithin{equation}{section}

\newcommand{\A}{\mathbb{A}}

\newcommand{\Q}{\mathbb{Q}}
\newcommand{\R}{\mathbb{R}}

\newcommand{\Z}{\mathbb{Z}}

\newcommand{\mcO}{\mathcal{O}}

\newcommand{\mfa}{\mathfrak{a}}
\newcommand{\mfb}{\mathfrak{b}}
\newcommand{\mfc}{\mathfrak{c}}

\newcommand{\h}{\mathfrak{h}}

\newcommand{\mfm}{\mathfrak{m}}
\newcommand{\mfn}{\mathfrak{n}}
\newcommand{\mfr}{\mathfrak{r}}
\newcommand{\Dif}{\mathfrak{D}}
\newcommand{\mfp}{\mathfrak{p}}

\newcommand{\SL}{\mathrm{SL}}
\newcommand{\GL}{\mathrm{GL}}
\newcommand{\SO}{\mathrm{SO}}

\def\coloneqq{\mathrel{\mathop:}=}%

\newcommand{\ra}{\rightarrow}

\def\1{1\!\!1}

\newcommand{\pmat}[4]{ \begin{pmatrix} #1 & #2 \\ #3 & #4 \end{pmatrix}}

\def\dis{\displaystyle}

\title[Equidistribution of signs for HMF of half-integral weight]{Equidistribution of signs for Hilbert modular forms of half-integral weight}

\author[S. Kaushik]{Surjeet Kaushik}
\address[S. Kaushik]{Department of Mathematics, Indian Institute of Technology Hyderabad, Kandi, Sangareddy 502285, INDIA}
\email{amarsurjeetkaushik@gmail.com}

\author[N. Kumar]{Narasimha Kumar}
\address[N. Kumar]{Department of Mathematics, Indian Institute of Technology Hyderabad, Kandi, Sangareddy 502285, INDIA}
\email{narasimha.kumar@iith.ac.in}

\author[N. Tanabe]{Naomi Tanabe}
\address[N. Tanabe]{Department of Mathematics, Bowdoin College, Brunswick, Maine 04011, USA}
\email{ntanabe@bowdoin.edu}


\date{\today}

\begin{document}

\begin{abstract}
We prove an equidistribution of signs for the Fourier coefficients of Hilbert modular forms of half-integral weight. 
Our study focuses on certain subfamilies of coefficients that are accessible via the Shimura correspondence. 
This is a generalization of the result of Inam and Wiese \cite{IW13} to the setting of totally real number fields. 
\end{abstract}
\maketitle
             
\section{Introduction}
The Fourier coefficients of integral or half-integral weight modular forms over number fields have been extensively studied 
because of rich arithmetic and algebraic properties that they encompass. In recent years, many problems addressing the sign changes 
of these Fourier coefficients have been studied by various authors. 
In this article, we are interested in the equidistribution of signs of these Fourier coefficients.

For the integral weight cusp forms over $\Q$, one can show  that the Fourier coefficients  change signs infinitely often (cf.~\cite{Mur83}).
For normalized cuspidal eigenforms of integral weight without complex multiplication (CM), the equidistribution of signs is a consequence 
of the Sato-Tate equidistribution theorem due to Barnet-Lamb, Geraghty, Harris, and Taylor (cf.~\cite{BGHT11}).
Recently, study on sign change has been extended to cusp forms of half-integral weight over $\Q$ (cf.~\cite{Koh07}).
In~\cite{BK08}, Bruinier and Kohnen conjectured an equidistribution of signs for half-integral weight modular forms over $\Q$ (cf.~\cite{KLW13} for more details). 
In~\cite{IW13, IW16}, Inam and Wiese showed the equidistribution of signs for certain subfamilies of coefficients that are accessible via the Shimura correspondence.


It is natural to ask similar questions for modular forms defined over number fields, in particular, over totally real number fields, say $F$.
There are not many results available in this setting as compared to the classical case (over $\Q$). 
For Hilbert cusp forms of integral weight, one can show that the Fourier coefficients  change signs infinitely often (cf.~\cite{MT14}).
Furthermore,  
the equidistribution of signs for primitive Hilbert forms of integral weight without CM can be obtained as, similar to the case of classical forms, a consequence of the Sato-Tate equidistribution theorem due to Barnet-Lamb, Gee, and Geraghty \cite{BGG11} 
(cf. Theorem~\ref{equidistribution} in the text). To the best of authors knowledge, similar results are not available in the literature for Hilbert modular forms of half-integral weight. 

This article is a modest attempt to show that the ideas of Inam and Wiese in~\cite{IW13} generalize to the case of Hilbert modular cusp forms of half-integral weight. 
We show that the equidistribution of signs holds for certain subfamilies of coefficients that are accessible via the Shimura correspondence. The proof uses
the Sato-Tate equidistribution theorem for non-CM primitive Hilbert modular forms. As a consequence, we see that the Fourier coefficients  change signs infinitely often.

The article is organized as follows. In \S\ref{sec:preliminaries}, we review the basic definitions of Hilbert modular forms of integral, half-integral weight
and state the Shimura correspondence between them. In \S\ref{sec:Sato-Tate}, we recall the Sato-Tate equidistribution Theorem for
automorphic representations of $\GL_2(\mathbb{A}_F)$.  
In \S\ref{sec:main}, the main result of this article is stated as Theorem~\ref{Sato-Tate-signchanges} with its proof immediately after. 

\section{Preliminaries}\label{sec:preliminaries}
In this section, we recall important definitions and properties of holomorphic Hilbert modular forms of both integral and half-integral weight. 
We adopt the setting from Shimura~\cite{Shi87}. For the case of integral weight, the reader may also refer~\cite{RT11} and~\cite{Shi78}. 
Throughout the paper, let $F$ denote a totally real number field of degree $n$, $\mcO_F$ its ring of integers, and $\Dif_F$ the different ideal of $F$. 
We also fix the order of real embeddings $\{\eta_j\}_{j=1}^n$ of $F$ and denote as $\eta$. 
Then, an element $\alpha$ in $F$ sits inside $\R^n$ by $(\eta_1(\alpha),\dots,\eta_n(\alpha))$. 
We write it as $(\alpha_1,\dots,\alpha_n)$,  $\eta(\alpha)$, or simply $\alpha$ again when no confusion arises.
An element $\alpha\in F$ is said to be totally positive if $\alpha_j>0$ for all $j$, and for any subset $X\subset F$, we denote by $X^+$ the set of totally positive elements in $X$.

\subsection{Half-integral weight Hilbert modular forms}
\label{upperhalf-plane-version}
Let $k=(k_1,\dots,k_n)$ be an integral or a half-integral weight, i.e., $k_j\in\Z_{>0}$ for all $j$ or $k_j=1/2+m_j$ with $m_j\in\Z_{>0}$ for all $j$, respectively. Both cases together, we denote $k=u/2+m$ while it is understood that $u\in\{0, 1\}$ and $m=(m_1,\dots, m_n)\in\Z_{>0}^n$. Given a holomorphic function $g$ on $\h^n$ and an element $\dis \gamma=(\gamma_1,\dots,\gamma_n)$ of $\SL_2(\R)^n$ with $\dis \gamma_j=\left(\begin{matrix} a_j & b_j \\ c_j & d_j \end{matrix}\right)$, define 
\begin{equation}\label{eqn:k_stroke}
 g|\!|_k \gamma (z)=h(\gamma, z)^{-u}\prod_j (c_jz_j+d_j)^{-m_j} g(\gamma z)
\end{equation}
where $z\in \h^n$ 
and $h(\gamma, z)$ is some non-vanishing holomorphic function on $\h^n$. See \cite[Proposition 2.3]{Shi87} for the precise definition for $h$. We note that the function $h$ is only defined when $\gamma$ is in a ``nice" subgroup of $\SL_2(\R)$, but we shall not worry about the details as we only consider such a congruence subgroup $\Gamma$ of $\SL_2(\R)$ as in (\ref{eqn:cong}). 
We refer the reader to Shimura \cite[Section 2]{Shi87} for the details.

For the rest of this section, we assume $k$ is half-integral. Let $\mfc$ be an integral ideal of $F$ that is divisible by $4$, and define a congruence subgroup $\Gamma=\Gamma(\mfc)$ of $\SL_2(\R)$ by
\begin{equation}\label{eqn:cong}
 \Gamma(\mfc)=\left\{ \left(\begin{matrix} a & b \\ c & d \end{matrix}\right)\in \SL_2(\R)\, :\, \begin{matrix} a\in\mcO_F, & b\in 2\Dif_F^{-1}  \\ c\in 2^{-1}\mfc\Dif_F, & d\in \mcO_F \end{matrix} \right\}.
\end{equation}
Let us take a Hecke character $\psi$ on the idele group $\A_F^\times$ of $F$ whose conductor divides $\mfc$ and infinite part $\psi_\infty=\prod_j \psi_{\eta_j}$ satisfies the following condition;
\begin{equation}\label{eq:psi}
\psi_\infty(-1)=(-1)^{\sum_j m_j}.
\end{equation}
Such a character $\psi$ can be extended to a character of $\Gamma(\mfc)$, which is again denoted by $\psi$, as $\psi(\gamma)=\psi(a)$ where $\dis \gamma =\left(\begin{matrix} a & b \\ c & d \end{matrix}\right)$. 
We denote by $M_k(\Gamma(\mfc),\psi)$ the set of all holomorphic functions $g$ on $\h^n$ satisfying
\[ g|\!|_k\gamma=\psi_\mfc(\gamma)g\]
for all $\gamma\in \Gamma(\mfc)$, where $\psi_\mfc$ is the ``$\mfc$-part" of $\psi$, i.e., $\psi_\mfc=\prod_{\mfp|\mfc}\psi_\mfp$. 
It should be noted that our choice of Hecke character $\psi$ only depends on finitely many places, namely at archimedean places 
and the $\mfc$-part. However, we will keep the notation $M_k(\Gamma(\mfc),\psi)$ without replacing $\psi$ with $\psi_\mfc$, as the choice of characters become more crucial in \S\ref{sec:main}.

Such a form $g \in M_k(\Gamma(\mfc),\psi)$ is known to have the Fourier expansion corresponding to any given fractional ideal $\mfa$. Its coefficients are denoted by $\{\lambda_{g}(\xi, \mfa)\}_{\xi, \mfa}$ where $\xi$ varies over totally positive elements in  $F$. One can treat $\{\lambda_{g}(\xi,\mfa)\}$ as a two parameter family of Fourier coefficients for $g$, as varies 
over $\xi$ in $F$ and fractional ideals $\mfa$. It should be noted that $\lambda_{g}(\xi,\mfa)=0$ unless $\xi\in (\mfa^{-2})^+$ or $\xi=0$. 
A modular form $g$ is said to be a cusp form if $\lambda_{g|_{\gamma}}(0,\mfa) = 0$ 
for every fractional ideal $\mfa$ and every $\gamma \in \SL_2(F)$. The space of such $g$ is denoted by $S_k(\Gamma(\mfc),\psi)$. 
For more details, we refer the reader to~\cite[Proposition 3.1]{Shi87}.

Our aim in this article is to study the equidistribution of signs for a family of Fourier coefficients $\{ \lambda_g(\xi,\mfa) \}$
with $\mfa$ varying over a certain family of fractional ideals.


\subsection{Integral weight Hilbert modular forms} This section is essentially a summary of Raghuram and Tanabe \cite[Section 4.1]{RT11} with some modifications following Shimura \cite[Section 6]{Shi87}.  

We assume that $k=(k_1,\dots, k_n)\in \Z_{>0}^n$ throughout this section. For a non-archimedean place $\mfp$ of $F$, 
let $F_\mfp$ be a completion of $F$. Let $\mfa$ and $\mfb$ be integral ideals of $F$, and define a subgroup $K_\mfp(\mfa, \mfb)$ of $\GL_2(F_\mfp)$ as
\[ K_\mfp(\mfa, \mfb)=\left\{\left(\begin{matrix} a & b \\ c & d \end{matrix}\right)\in \GL_2(F_\mfp)\, : \, \begin{matrix} a\in \mcO_\mfp, & b\in \mfa_\mfp^{-1}\Dif_\mfp^{-1}, & \\ c\in \mfb_\mfp\Dif_\mfp, & d\in\mcO_\mfp, & |ad-bc|_\mfp=1\end{matrix}\right\}\]
where the subscript $\mfp$ means the $\mfp$-parts of given ideals. Furthermore, we put 
\[ K_0(\mfa, \mfb)=\SO(2)^n\cdot\prod_{\mfp<\infty}K_\mfp(\mfa, \mfb) \quad \text{and} \quad W(\mfa, \mfb)=\GL_2^+(\R)^nK_0(\mfa, \mfb).\]
In particular, if $\mfa=\mcO_F$, we simply write $K_\mfp(\mfb)=K_\mfp(\mcO_F, \mfb)$ and $W(\mfb)=W(\mcO_F, \mfb)$.
Then, we have the following disjoint decomposition of $\GL_2(\A_F)$,
\begin{equation}\label{eqn:decomp}
 \GL_2(\A_F)=\cup_{\nu=1}^h\GL_2(F)x_\nu^{-\iota} W(\mfb),
\end{equation}
where $\dis x_\nu^{-\iota} =\left(\begin{matrix} t_\nu^{-1} & \\ & 1\end{matrix}\right)$ with $\{t_\nu\}_{\nu=1}^h$ taken to be a complete set of representatives of the narrow class group of $F$. We note that such $t_\nu$ can be chosen so that the infinity part $t_{\nu, \infty}$ is $1$ for all $\nu$. For each $\nu$, we also put
\begin{align*}
\Gamma_\nu(\mfa, \mfb) &= \GL_2(F)\cap x_\nu W(\mfa, \mfb)x_\nu^{-1} \\ 
&= \left\{ \pmat{a}{t_\nu^{-1}b}{t_\nu c}{d}\in\GL_2(F): \, \begin{matrix} a\in \mcO_F, & b\in\mfa^{-1}\Dif_F^{-1}, & \\ c\in \mfb\Dif_F, & d\in\mcO_F, & ad-bc\in \mcO_F\end{matrix}\right\}. 
\end{align*} 
It is understood that $\Gamma_\nu(\mfb)=\Gamma_\nu(\mcO_F, \mfb)$ as before. 

Let $\psi$ be a Hecke character of $\A_F^\times$ such that its conductor divides $\mfb$ and its infinite part $\psi_\infty$ is of the form
\[ \psi_\infty(x)={\rm sgn}(x_\infty)^k|x_\infty|^{i\mu}\]
where $\mu=(\mu_1,\dots, \mu_n)\in\R^n$ with $\sum_{j=1}^n \mu_j=0$. We let $M_k(\Gamma_\nu(\mfb), \psi_\mfb, \mu)$ denote the space of all functions $f_\nu$ that are holomorphic on $\h^n$ and 
at cusps, satisfying
\[ f_\nu |\!|_k \gamma=\psi_\mfb(\gamma)\det \gamma^{i\mu/2}f_\nu \]
for all $\gamma$ in $\Gamma_\nu(\mfb)$. We note that such a function $f_\nu$ affords a Fourier expansion, and its coefficients are denoted as $\{a_\nu(\xi)\}_\xi$
where $\xi$ runs over all the totally positive elements in $t_\nu^{-1}\mcO_F$ and $\xi=0$. 
Similar to the case of half-integral weight forms, a Hilbert modular form is called a cusp form if, for all $\gamma \in \GL^+_2(F)$, the constant term of $f|\!|_k\gamma$
in its Fourier expansion is $0$, and the space of cusp forms with respect to $\Gamma_{\nu}(\mfb)$ is denoted by $S_k(\Gamma_\nu(\mfb), \psi_\mfb, \mu)$.

Now, put $\mathbf{f}=(f_1,\dots,f_h)$ where $f_\nu$ belongs to $M_k(\Gamma_\nu(\mfb), \psi_\mfb, \mu)$ for each $\nu$, and define $\mathbf{f}$ to be a function on $\GL_2(\A_F)$ as
\begin{equation}\label{eq:HMF}
\mathbf{f}(g)=\mathbf{f}(\gamma x_\nu^{-\iota}w)\coloneqq\psi_\mfb(w^\iota)\det w_\infty^{i\mu/2}(f_\nu||_k w_\infty)(i\!\! i)\end{equation}
where $\gamma x_\nu^{-\iota}w\in\GL_2(F)x_\nu^{-\iota}W(\mfb)$ as in (\ref{eqn:decomp}), $w^\iota=\omega_0(^t w)\omega_0^{-1}$ with $\dis \omega_0=\left(\begin{matrix} & 1 \\ -1 & \end{matrix}\right)$, and $i=(i,\dots, i)$. 
The space of such $\mathbf{f}$ is denoted as 
\[ \mathfrak{M}_k(\psi_\mfb, \mu)=\prod_\nu M_k(\Gamma_\nu(\mfb), \psi_\mfb, \mu).\] 
Furthermore, the space consisting of all  $\mathbf{f}=(f_1,\dots,f_h)\in \mathfrak{M}_k(\psi_\mfb, \mu)$ satisfying
\[ \mathbf{f}(xg)=\psi(x)\mathbf{f}(g) \quad \text{for any}\, x\in \A_F^\times \quad \text{and}\quad g\in\GL_2(\A_F)\]
is denoted as $\mathfrak{M}_k(\mfb,\psi)$.  In particular, if each $f_\nu$ belongs to $S_k(\Gamma_\nu(\mfb), \psi_\mfb, \mu)$, then the space of such $\mathbf{f}$ is denoted by $\mathfrak{S}_k(\mfb,\psi)$.
A cusp form $\mathbf{f}$ is called primitive if it is a normalized new form and a common eigenfunction of all Hecke operators.

Let $\mfm$ be an ideal of $F$ and write $\mfm=\xi t_\nu^{-1}\mcO_F$ with a totally positive element $\xi$ in $F$. Then the Fourier coefficient of $\mathbf{f}$ at $\mfm$ is defined as 
\begin{equation}\label{eqn:coeff}
c(\mfm,\mathbf{f})=\begin{cases} a_\nu(\xi)\xi^{-(k+i\mu)/2} \quad & \text{if}\quad \mfm=\xi t_\nu^{-1}\mcO_F\subset \mcO_F \\
0 & \text{if} \quad \mfm \,\, \text{is not integral}. \end{cases}
\end{equation}






\subsection{Shimura Correspondence}
 In this section, we shall recall the Shimura correspondence, which states that, 
given a non-zero half-integral weight cusp form, there is an integral automorphic form associated with it.

For any integral ideal $\mfa$ in $\mcO_F$, we introduce a formal symbol ${\rm M}(\mfa)$ satisfying that ${\rm M}(\mcO_F)=1$ and ${\rm M}(\mfa \mfb) = {\rm M}(\mfa){\rm M}(\mfb)$ for all $\mfa,\mfb \subseteq \mcO_F$.
Then one can consider the ring of formal series in these symbols, indexed by integral ideals.

The following result is the Shimura correspondence for Hilbert modular forms~\cite[Theorems 6.1 and 6.2]{Shi87}. 
  We assume for simplicity that $\psi$ is a quadratic character, as it will be the case in our setting.
 \begin{thm}\label{Shimuralift}
 Let $ 0 \neq g\in S_{k}(\Gamma(\mfc),\psi)$ with a half-integral weight $k=\frac{1}{2}+m$  with $m \geq 1$, an integral ideal $\mfc$ of $F$ divisible by $4$,
 and $\psi$ being a Hecke character of $F$ such that it satisfies \eqref{eq:psi} and its conductor divides $\mfc$. Further, we assume that $\psi_\infty(x)=|x|^{i\mu}$ for any totally positive element 
 $x$ in $\A_{F,\infty}^\times$ with some $\mu\in\R^n$ such that $\sum_j\mu_j=0$. 

 Let $\tau$ be an arbitrary element in $\mcO_F^+$, and 
 write $\tau\mcO_F=\mfa^2\mfr$ for some integral ideal $\mfa$ and a square free integral ideal $\mfr$. Then the following assertions hold.
\begin{enumerate}
\item \label{one}
      Let $\mfb$ be a fractional ideal of $F$ and define $\Gamma=\GL_2(F)\cap W(\mfb,2^{-1}\mfc\mfb)$. Then there exists $f\in M_{2m}(\Gamma,\psi_{2^{-1}\mfc}^2,\mu)$ so that 
      \begin{equation}\label{first}
       \sum_{\xi\in\mfb/\mcO_{F}^{\times, +}} a(\xi)\xi^{-m-i\mu}{\rm M}(\xi\mfb^{-1})=\sum_{\mfm\subseteq \mcO_F} \lambda_g(\tau,\mfa^{-1}\mfm){\rm M}(\mfm)
	\sum_{\mfn\mfm\mfb \sim 1}(\psi{\epsilon_\tau})^*(\mfn){\rm N}(\mfn)^{-1}{\rm M}(\mfn),
      \end{equation} 
       where $\lambda_g(\tau,\mfa^{-1}\mfm)$ is the Fourier coefficient of $g$ at a cusp corresponding to 
       $\mfa^{-1}\mfm$, $\psi_{ \{ 2^{-1}\mfc \} }= \prod_{\mfp|2^{-1} \mfc }\psi_\mfp$, $\epsilon_\tau$ is the Hecke character of $F$
       corresponding to $F(\sqrt{\tau})/F$, and $(\psi{\epsilon_\tau})^*$ is the character induced from $\psi{\epsilon_\tau}$. In the second sum of the right hand side in (\ref{first}), $\mfn$ runs through all the integral ideal of $F$ that are prime to $\mfc\mfr$ and equivalent to $(\mfm\mfb)^{-1}$. 
 		
 \item \label{two} Let $\mathbf{f}_\tau=(f_1,\dots,f_h)$ where $f_\nu$ is of the form given in~\eqref{one} with $\mfb=t_\nu\mcO_F$ for all $\nu=1,\dots,h$. 
 Then $\mathbf{f}_\tau\in \mathfrak{M}_{2m}(2^{-1}\mfc,\psi^2)$ and it satisfies
 	\begin{equation} \label{eqn:adelized}
 	\sum_\mfm c(\mfm,\mathbf{f}_\tau){\rm M}(\mfm)
 	=\sum_{\mfm}\lambda_g(\tau,\mfa^{-1}\mfm){\rm M}(\mfm)\prod_{\mfp\nmid\mfc\mfr}\left(1-\frac{(\psi{\epsilon_\tau})^*(\mfp)} {{\rm N}(\mfp)} {\rm M}(\mfp)\right)^{-1}
 	\end{equation}
        where  $\mfm$ runs over all the integral ideals of $F$ and $\mfp$ over all the prime ideals which do not divide $\mfc\mfr$.
        
 \item The function $f$ given in {\rm (1)} is a cusp form if $m_j > 1$ for some $j$. 
 		
\end{enumerate}
 \end{thm}

\section{Sato-Tate equidistribution theorem for Hilbert Modular forms}
\label{sec:Sato-Tate}
We now recall the Sato-Tate equidistribution theorem for Hilbert modular forms of integral weight without CM. 
Let $\mathbb{P}$ denote the set of all prime ideals of $\mcO_F$. We begin by recalling the notion of natural density for a subset of $\mathbb{P}$.

\begin{dfn}
	Let $F$ be a number field and $S$ be a subset of $\mathbb{P}$. We define the natural density of $S$ to be 
	\begin{equation}
	d(S) = \underset{x \ra \infty}{\mathrm{lim}}\  
	\frac{\# \{ \mfp  : {\rm N}(\mfp) \leq x, \mfp \in S \}}{\# \{   \mfp : {\rm N}(\mfp) \leq x, \mfp \in \mathbb{P} \}},
	\end{equation}
	provided the limit exists. 
\end{dfn}

In~\cite{BGG11}, 
Barnet-Lamb, Gee, and Geraghty proved the following Sato-Tate equidistribution theorem.
\begin{thm}{\rm (Barnet-Lamb, Gee, Geraghty, \cite[Corollary 7.1.7]{BGG11})}
\label{integralsato-tate}
Let $F$ be a totally real number field of degree $n$ and $\Pi$ a non-CM regular algebraic cuspidal automorphic representation of $\GL_2(\A_F)$. 
Write $\mu=(\mu_1,\dots,\mu_n)$ for an integral weight for the diagonal torus of $\GL_2(\R)^n$ with $\mu_j=(a_j, b_j)$ and $a_j\geq b_j$ for all $j$. We note that the values $a_j+b_j$ are the same for all $j$, and therefore we may put $\omega_\Pi=a_j+b_j$.
Let $\chi$ be the product of the central character of $\Pi$ with $|\cdot|^{\omega_{\pi}}$, so that $\chi$ is a finite order character. 
Let $\zeta$ be a root of unity such that $\zeta^2$ is in the image of $\chi$. 
For any place $\mfp$ of $F$ such that $\Pi_{\mfp}$ is unramified, let $\lambda_\mfp$ denote the eigenvalue of the Hecke operator
\begin{equation}\label{Heckeoperator}
 \GL_2(\mcO_\mfp)\left(\begin{matrix} \varpi_\mfp & \\ & 1 \end{matrix}\right) \GL_2(\mcO_\mfp)
\end{equation}
on $\Pi_{\mfp}^{\GL_2(\mcO_{\mfp})}$, where $\varpi_\mfp$ is a uniformizer of $\mcO_\mfp$.

Then as $\mfp$ ranges over the unramified places of $F$ 
such that $\chi_\mfp(\varpi_\mfp)=\zeta^2$, the number given by
\[ \frac{\lambda_\mfp}{2{\rm N}(\mfp)^{(1+\omega_\Pi)/2}\zeta}\]
belongs to $[-1,1]$, and furthermore they are equidistributed in $[-1,1]$ with respect to the measure $(2/\pi)\sqrt{1-t^2}dt$.
\end{thm}

For convenience, we now rewrite the above theorem in terms of 
the Fourier coefficients of primitive Hilbert modular forms.
From now on, we assume that $\mathbf{f}$ is a primitive form of weight $k=(k_1, \dots, k_n)$ with $k_1\equiv \dots \equiv k_n \equiv 0 \pmod{2}$ and each $k_j \geq 2$, level $\mfc$, 
and with trivial nebentypus. We first need the following lemma:

\begin{lem}  \label{lem:relation}
Let $\mathbf{f} \in \mathfrak{S}_k(\mfc,\1)$ be a primitive form, that is, a normalized common eigenform in the new space, and $\Pi=\Pi_\mathbf{f}$ an irreducible cuspidal automorphic representation corresponding to $\mathbf{f}$.
Let $\mfp$ be an prime ideal of $F$ such that $\mfp \nmid \mfc\Dif_{F}$.  
Let $c(\mfp,\mathbf{f})$ be the Fourier coefficients at $\mfp$ defined as in ~\eqref{eqn:coeff},
and $\lambda_{\mfp}$ be the eigenvalue of the Hecke operator defined as in ~\eqref{Heckeoperator}, then 
$$\lambda_\mfp=c(\mfp,\mathbf{f}){\rm N}(\mfp).$$ 
\end{lem}
\begin{proof}
The proof of this lemma can be found in~\cite[Page 305-306]{RT11}.
\end{proof}

To make the calculation simpler at a later point, we re-normalize $c(\mfp,\mathbf{f})$ as follows:
Define $$C(\mfp,\mathbf{f})=c(\mfp,\mathbf{f}){\rm N}(\mfp)^{k_0/2},$$
where 
$k_0=\max_j\{k_j\}$ with $k=(k_1, \dots, k_n)$ being the weight of $\mathbf{f}$.
Since we are interested in sign changes of the Fourier coefficients, this re-normalization  will not make any difference.
The following theorem is a consequence of Theorem~\ref{integralsato-tate} above:
\begin{thm}
\label{equidistribution}
Let $\mathbf{f}\in \mathfrak{S}_k(\mfc,\1)$ be a primitive form of weight $k=(k_1, \dots, k_n)$ such that $k_1\equiv \dots \equiv k_n \equiv 0 \pmod{2}$ and each $k_j \geq 2$.
Suppose that $\mathbf{f}$ does not have complex multiplication. Then, for any prime ideal $\mfp$ of $F$ such that $\mfp\nmid \mfc\Dif_F$, 
we have
$$ B(\mfp)\coloneqq\frac{C(\mfp,\mathbf{f})}{2 {\rm N}(\mfp)^{\frac{k_0-1}{2}}}\in[-1,1]. $$
Furthermore, $\{B(\mfp)\}_\mfp$ are equidistributed in $[-1,1]$ with respect to the measure $\mu=(2/\pi) \sqrt{1-t^2}dt$. 
In other words, for any subinterval $I$ of $[-1,1]$, we have
$$\underset{x \ra \infty}{{\lim}}\  
\frac{\# \{ \mfp\in \mathbb{P} : \mfp\nmid \mfc\Dif_F, {\rm N}(\mfp) \leq x, B(\mfp) \in I \}}{\# \{   \mfp\in \mathbb{P} : {\rm N}(\mfp) \leq x \}}
= \mu(I)= \frac{2}{\pi} \int_I \sqrt{1-t^2} dt,$$
i.e., the natural density of the set $\{ \mfp: B(\mfp) \in I\}$ is $\mu(I)$.
\end{thm}

\begin{proof}
Let $\Pi=\Pi_{\mathbf{f}}$ be the non-CM irreducible cuspidal automorphic representation corresponding to $\mathbf{f}$. 
Since $k_1\equiv \dots \equiv k_n \equiv 0 \pmod 2$ and each $k_j \geq 2$, it follows from \cite[Theorem 1.4]{RT11} that $\Pi$ is algebraic and regular.
 By Lemma~\ref{lem:relation}, we have
$$\frac{\lambda_\mfp}{2{\rm N}(\mfp)^{(1+\omega_\Pi)/2}}=\frac{{\rm N}(\mfp)c(\mfp,\mathbf{f})}{2{\rm N}(\mfp)^{(1+\omega_\Pi)/2}}
=\frac{c(\mfp,\mathbf{f})}{2{\rm N}(\mfp)^{(\omega_\Pi-1)/2}}.$$
Given that the highest weight vector $\mu = (\mu_1, \ldots, \mu_n)$ of $\Pi$ being $\mu_j=((k_j-2)/2,-(k_j-2)/2)$ (cf. \cite[Section 4.6]{RT11}), 
we have $\omega_{\Pi}=\frac{k_j-2}{2}-\frac{k_j-2}{2}=0$. Hence, we get
$$\frac{\lambda_\mfp}{2{\rm N}(\mfp)^{1/2}}=\frac{c(\mfp,\mathbf{f})}{2{\rm N}(\mfp)^{\frac{-1}{2}}} = \frac{C(\mfp,\mathbf{f})}{2{\rm N}(\mfp)^{\frac{k_0-1}{2}}}.$$
Now, the theorem follows from Theorem~\ref{integralsato-tate} by taking $\zeta=1$.
\end{proof}

\section{Main Result}\label{sec:main} 
Before stating the main theorem, we introduce some more notation. 
%
Let $\tau\in \mcO_F^+$ and $\mfa$ an integral ideal of $F$. Our interest is to study a certain family of Fourier coefficients of a half-integral cusp form $g$, namely $\{\lambda_g(\tau, \mfa^{-1}\mfp)\}_\mfp$ where $\mfp$ varies over prime ideals. 
For a fixed $g$, we put
$$\mathbb{P}_{>0}(\tau,\mfa)=\{ \mfp \in \mathbb{P} : \mfp \nmid \mfc\Dif_F,\  \lambda_g(\tau,\mfa^{-1}\mfp)>0 \},$$
	and similarly $\mathbb{P}_{<0}(\tau, \mfa)$, $\mathbb{P}_{\geq 0}(\tau, \mfa)$, $\mathbb{P}_{\leq 0}(\tau, \mfa)$, 
	and $\mathbb{P}_{=0}(\tau, \mfa)$.
We also write $\mathbb{P}_{\mfc}$ for the set of all prime ideals not dividing $\mfc$. 

We are now ready to state the main result of the article. 

\begin{thm}
\label{Sato-Tate-signchanges}
	
 Let $ 0 \neq g\in S_{k}(\Gamma(\mfc),\psi)$ with a half-integral weight $k=\frac{1}{2}+m$  with $m_j > 1$ for some $j$, an integral ideal $\mfc$ of $F$ divisible by $4$,
 and $\psi$ a Hecke character of $\A_F$ satisfying the following conditions:
 \begin{enumerate}
 \item[a.] the conductor of $\psi$ divides $\mfc$, 
 \item[b.] $\psi_\infty(-1)=(-1)^{\sum_jm_j}$, and
 \item[c.] for any totally positive element $x$ in $\A_{F,\infty}^\times$, $\psi_\infty(x)=|x|^{i\mu}$  with some $\mu\in\R^n$ such that $\sum_j\mu_j=0$.  

 \end{enumerate}
Furthermore, we suppose that the Fourier coefficients of $g$ are real and the character $\psi$ of $g$ is  quadratic.

 Let $\tau$ be an arbitrary element in $\mcO_F^+$, and  write $\tau\mcO_F=\mfa^2\mfr$ for some integral ideal $\mfa$ and a square free integral ideal $\mfr$.
Then, there is a lift $\mathbf{f}_{\tau}$ of $g$ under the Shimura correspondence 
 (as in Theorem ~\ref{Shimuralift}). Assume that $\mathbf{f}_{\tau}$ is a non-CM primitive Hilbert modular form.

 Then, the natural density of $\mathbb{P}_{>0}(\tau, \mfa)$ $($resp., of $\mathbb{P}_{<0}(\tau, \mfa)$ $)$ is $1/2$, i.e., $d(\mathbb{P}_{>0}(\tau, \mfa)) =1/2$ $($resp., $d(\mathbb{P}_{<0}(\tau, \mfa)) =1/2$ $)$,
 and $d(\mathbb{P}_{=0}(\tau, \mfa))=0$.  
\end{thm}

The rest of this article is devoted to proving the theorem.  From now on, we simply write $\mathbb{P}_{>0}$ for $\mathbb{P}_{>0}(\tau, \mfa)$, etc. 
Let us also define 
\[\pi(x) = \#\{\mfp \in \mathbb{P} :  {\rm N}(\mfp)\leq x \} \quad  \text{and}\quad
       \pi_{>0}(x)=\#\{\mfp \in\mathbb{P}_{>0} : {\rm N}(\mfp)\leq x\}.\]
Then we have the following proposition.
 

\begin{prop}
\label{equ:liminf}
Assume that all the hypotheses in Theorem~\ref{Sato-Tate-signchanges} hold. Then, we have
 \[\liminf_{x\rightarrow\infty}\frac{\pi_{>0}(x)}{\pi(x)}\geq\mu([0,1])=\frac{1}{2} \quad \text{and}\quad
 \liminf_{x\rightarrow\infty}\frac{\pi_{\leq 0}(x)}{\pi(x)}\geq\mu([0,1])=\frac{1}{2}.\]
\end{prop} 
   
\begin{proof}
The equality of the formal sums in \eqref{eqn:adelized} can also be re-interpreted as
\begin{equation*}
\prod_{\mfp\nmid\mfc\mfr}\left(1-\frac{(\psi{\epsilon_\tau})^*(\mfp)}{{\rm N}(\mfp)}{\rm M}(\mfp)\right)\sum_\mfm c(\mfm,\mathbf{f}_\tau){\rm M}(\mfm)=\sum_{\mfm}\lambda_g(\tau,\mfa^{-1}\mfm){\rm M}(\mfm).
\end{equation*}
For any non-zero prime ideal $\mfp \nmid \mfc\mfr$, comparing the coefficients of ${\rm M}(\mfp)$ on both sides gives
\begin{equation}
\label{relation}
c(\mfp,\mathbf{f}_\tau)-\frac{(\psi{\epsilon_\tau})^*(\mfp)}{{\rm N}(\mfp)}=\lambda_g(\tau,\mfa^{-1}\mfp)
\end{equation}
since $\mathbf{f}_{\tau}$ is primitive (i.e., $c(\mcO_F,\mathbf{f}_{\tau})=1$). There are exactly two terms on the left side of~\eqref{relation}
because the unique factorization of ideals holds in $\mcO_F$, and hence the only integral ideals which divides $\mfp$ are $\mfp$ and $\mcO_F$ itself.


Since $\psi$ is a quadratic character, the primitive Hilbert modular form $\mathbf{f}_{\tau}$ has trivial nebentypus.
Hence, the Fourier coefficients $c(\mfp, \mathbf{f}_{\tau})$ are real numbers (cf.~\cite[Proposition 2.5]{Shi78}). 
This implies that $(\psi{\epsilon_\tau})^*(\mfp) \in \{ \pm 1\}$ since, by our assumption, $(\psi{\epsilon_\tau})^*(\mfp)$ is a root of unity 
and $\lambda_g(\tau, \mfm)$ is real for all fractional ideals $\mfm$.
 
By~\eqref{relation}, we have
 $$\lambda_g(\tau,\mfa^{-1}\mfp)>0\Leftrightarrow c({\mfp},\mathbf{f}_\tau)>\frac{(\psi{\epsilon_\tau})^*(\mfp)}{{\rm N}({\mfp})},$$ 
which gives us $$\lambda_g(\tau,\mfa^{-1}\mfp)>0 \Leftrightarrow B(\mfp)>\frac{(\psi{\epsilon_\tau})^*(\mfp)}{2{\rm N}({\mfp})^{\frac{1}{2}}}$$
since $B(\mfp)=\frac{C(\mfp,\mathbf{f}_{\tau})}{2{\rm N}(\mfp)^{\frac{k_0-1}{2}}}$.  

For any $\epsilon >0$, we have the following inequality:
$$\pi_{>0}(x)+\pi\left(\frac{1}{4\epsilon^2}\right)\geq \#\{ \mfp \in\mathbb{P}_{\mfc\mfr \Dif_F} : {\rm N}(\mfp)\leq x \ \mathrm{and}\  B(\mfp)>\epsilon\}$$
since $|\frac{(\psi{\epsilon_\tau})^*(\mfp)}{2N({\mfp})^{1/2}}|=\frac{1}{2N({\mfp})^{1/2}}<\epsilon$ if ${\rm N}(\mfp)>{1}/4{\epsilon^{2}}.$
Now divide the above inequality by $\pi(x)$ to obtain
$$\frac{\pi_{>0}(x)}{\pi(x)}+\frac{\pi\left(\frac{1}{4\epsilon^2}\right)}{\pi(x)}\geq\frac{\#\{ \mfp \in\mathbb{P}_{\mfc\mfr \Dif_F} : {\rm N}(\mfp)\leq x \ \mathrm{and}\  B(\mfp)>\epsilon\}}{\pi(x)}.$$
The term 
$\pi\left(\frac{1}{4\epsilon^2}\right)/\pi(x)$ tends to $0$, as $x\rightarrow\infty$, since $\pi\left(\frac{1}{4\epsilon^2}\right)$ is finite. 
On the other hand, Theorem~\ref{equidistribution} gives
$$\frac{\#\{ \mfp \in\mathbb{P} : {\rm N}(\mfp)\leq x \ \mathrm{and}\  B(\mfp)>\epsilon \}}{\pi(x)}\rightarrow\mu([\epsilon,1])$$ as $x\rightarrow\infty$,
and therefore we have $$\liminf_{x\rightarrow\infty}\frac{\pi_{>0}(x)}{\pi(x)}\geq\mu([\epsilon,1]) \ \mathrm{for \ all} \ \epsilon>0.$$ 
Hence, we can conclude that 
$$\liminf_{x\rightarrow\infty}\frac{\pi_{>0}(x)}{\pi(x)}\geq\mu([0,1])=\frac{1}{2}.$$
A similar proof shows that 
$$\liminf_{x\rightarrow\infty}\frac{\pi_{\leq 0}(x)}{\pi(x)}\geq\mu([0,1])=\frac{1}{2}.$$

\end{proof}

\begin{proof}[\textbf{Proof of Theorem of~\ref{Sato-Tate-signchanges}}]
By Proposition ~\ref{equ:liminf}, we have
$$\frac{1}{2} \leq \liminf_{x\rightarrow\infty}\frac{\pi_{>0}(x)}{\pi(x)}.$$
Since $\pi_{>0}(x)=\pi(x)-\pi_{\leq 0}(x)$, we have 
$$\limsup_{x\rightarrow\infty}\frac{\pi_{>0}(x)}{\pi(x)}\leq\mu([0,1])=\frac{1}{2}.$$ 
Hence,
$$\frac{1}{2} \leq \liminf_{x\rightarrow\infty}\frac{\pi_{>0}(x)}{\pi(x)} \leq \limsup_{x\rightarrow\infty}\frac{\pi_{>0}(x)}{\pi(x)}\leq\mu([0,1])=\frac{1}{2},$$
and therefore, $\lim_{x\rightarrow\infty}\frac{\pi_{>0}(x)}{\pi(x)}$ exists and equals $\frac{1}{2}$. 
Thus, the set $\mathbb{P}_{>0}$ has natural density $\frac{1}{2}$. The same argument yields that
$\mathbb{P}_{<0}$ has natural density $\frac{1}{2}$ as well. This proves that $\mathbb{P}_{=0}$ has natural density $0$.
\end{proof}
We conclude the article with the following corollary.
\begin{cor}
Assume that all the hypotheses of Theorem~\ref{Sato-Tate-signchanges} hold.
Then, the set $\{\lambda_g(\tau, \mfa^{-1}\mfp)\}_{\mfp \in \mathbb{P}}$ changes signs infinitely often.
In particular, there exist infinitely many primes $\mfp \in \mathbb{P}$ for which $\lambda_g(\tau,\mfa^{-1}\mfp)>0$ (resp., $\lambda_g(\tau,\mfa^{-1}\mfp)<0$).
\end{cor}


\section*{Acknowledgements}
The authors would like to express their gratitude to an anonymous referee for his/her careful reading of our manuscript and their suggestions. 
The research of the second author was supported partly by the start up research grant of IIT Hyderabad.

\bigskip \bigskip


\bigskip


\begin{thebibliography}{100}
\bibitem{BGHT11} Barnet-Lamb, Thomas; Geraghty, David; Harris, Michael; Taylor, Richard. 
                 A family of Calabi-Yau varieties and potential automorphy II. 
                 Publ. Res. Inst. Math. Sci. 47 (2011), no. 1, 29--98.


\bibitem{BGG11} Barnet-Lamb, Thomas; Gee, Toby; Geraghty, David. 
		The Sato-Tate conjecture for Hilbert modular forms. 
		J. Amer. Math. Soc. 24 (2011), no. 2, 411--469.

\bibitem{BK08}  Bruinier, Jan Hendrik; Kohnen, Winfried. 
                Sign changes of coefficients of half integral weight modular forms. 
                Modular forms on Schiermonnikoog, 57--65, Cambridge Univ. Press, Cambridge, 2008.
		      
\bibitem{IW13}  Inam, Ilker; Wiese, Gabor. 
                Equidistribution of signs for modular eigenforms of half integral weight. 
                Arch. Math. (Basel) 101 (2013), no. 4, 331--339.
                
\bibitem{IW16} Inam, Ilker; Wiese, Gabor. 
		A short note on the Bruiner-Kohnen sign equidistribution conjecture and Hal\'asz' theorem.
		Int. J. Number Theory, 12 (2016), no. 2, 357--360.

\bibitem{Koh07} Kohnen, Winfried. 
                Sign changes of Fourier coefficients and eigenvalues of cusp forms. 
                Number theory, 97--107, Ser. Number Theory Appl., 2, World Sci. Publ., Hackensack, NJ, 2007. 
                
\bibitem{KLW13} Kohnen, W.; Lau, Y.-K.; Wu, J. 
                Fourier coefficients of cusp forms of half-integral weight. 
                Math. Z. 273 (2013), no. 1-2, 29--41.                
           
                            
                
\bibitem{Mur83} Murty, M. Ram. 
                Oscillations of Fourier coefficients of modular forms. 
                Math. Ann. 262 (1983), no. 4, 431--446.                
                
\bibitem{MT14}  Meher, Jaban; Tanabe, Naomi. 
                Sign changes of Fourier coefficients of Hilbert modular forms. 
                J. Number Theory 145 (2014), 230--244.	  
                
\bibitem{RT11}   Raghuram, A.; Tanabe, Naomi. 
		Notes on the arithmetic of Hilbert modular forms. 
		J. Ramanujan Math. Soc. 26 (2011), no. 3, 261--319.

\bibitem{Shi78} Shimura, Goro. 
                The special values of the zeta functions associated with Hilbert modular forms. 
	        Duke Math. J. 45 (1978), no. 3, 637--679.

\bibitem{Shi87} Shimura, Goro.  
                 On Hilbert modular forms of half-integral weight. 
		 Duke Math. J. 55 (1987), no. 4, 765--838.
         
\end{thebibliography}
\end{document}